\documentclass[11pt]{amsart}

\usepackage{amsmath}
\usepackage{amssymb}
\usepackage{graphicx}
\usepackage[
citestyle=numeric,
bibstyle=numeric
]{biblatex}
\usepackage{biblatex}
\usepackage{bbold}
\usepackage{hyperref}

\newtheorem{thm}{Theorem}[section]
\newtheorem*{thm*}{Theorem}
\newtheorem{prop}[thm]{Proposition}
\newtheorem{lem}[thm]{Lemma}
\newtheorem*{lem*}{Lemma}
\newtheorem*{cor*}{Corollary}
\theoremstyle{definition}
\newtheorem{defn}[thm]{Definition}
\newtheorem{cor}{Corollary}[section]

\newtheorem{conj}[thm]{Conjecture}
\newtheorem{rem}[thm]{Remark}

\numberwithin{equation}{section}

\title{Operator orbit Frames and frame-like Fourier expansions}
\author{Chad Berner and Eric S. Weber}

\begin{document}

\begin{abstract}
Frames in a Hilbert space that are generated by operator orbits are vastly studied because of the applications in dynamic sampling and signal recovery.
We demonstrate in this paper a representation theory for frames generated by operator orbits that provides explicit constructions of the frame and the operator when the operators are not surjective. It is known that the Kaczmarz algorithm for stationary sequences in Hilbert spaces generates a frame that arises from an operator orbit where the operator is not surjective.
In this paper, we show that every frame generated by a not surjective operator in any Hilbert space arises from the Kaczmarz algorithm. Furthermore, we show that the operators generating these frames are similar to rank one perturbations of unitary operators. After this, we describe a large class of operator orbit frames that arise from Fourier expansions for singular measures. Moreover, we classify all measures that possess frame-like Fourier expansions arising from two-sided operator orbit frames. Finally, we show that measures that possess frame-like Fourier expansions arising from two-sided operator orbits are weighted Lebesgue measure with weight satisfying a weak $A_{2}$ condition, even in the non-frame case. We also use these results to classify measures with other types of frame-like Fourier expansions.
\end{abstract}

\maketitle

\section{Introduction}

Consider the problem of recovering a vector $v$ in a Hilbert space $H$ from its inner-products with a fixed set of vectors $\{ \langle v , x_n \rangle \}$.  In the simplest case, we are describing the situation of solving a (finite) system of linear equations for which there are numerous algebraic and numerical techniques.  When $H$ is infinite dimensional, then more sophisticated ideas are required; some ideas include orthonormal or Riesz bases and frames, see \cite{Duffin1952Class},\cite{Casazza2000Art}.  These are special cases of situations in which there is a dual set $\{ y_{n} \}$ such that 
\[ v = \sum_{n} \langle v, x_n \rangle y_n \]
where the sum converges in a meaningful (though context dependent) manner.

We are motivated by the question of when the reconstruction of $v$ can be accomplished with Fourier series--here we are considering $H = L^2(\mu)$ for $\mu$ a measure on $\mathbb{R}$ and $y_{n} = y_{n}(x) := e^{2 \pi i x n}$.  Thus, we want to understand when $v \in L^2(\mu)$ can be reconstructed as
\[ v = \sum_{n} \langle v, x_{n} \rangle e^{2 \pi i n x}. \]
Here, we leave the range of $n$ ambiguous, since we will consider both $n \in \mathbb{N}_{0}=\mathbb{N}$ and $n \in \mathbb{Z}$.

Building on prior work, we will find that our question regarding Fourier series expansions leads to several intertwined topics:
\begin{enumerate}
    \item operator orbit frames \cite{Aldroubi2017Dynamical}
    \item the Kaczmarz algorithm \cite{Kaczmarz1993Approximate},\cite{Herr2017Fourier}
    \item frame-like duality in Fourier expansions.
\end{enumerate}
In our effort to describe Fourier expansions, we found that we needed a fuller description of operator orbit frames, see \cite{Christensen2020Frame}.  Indeed, we have both a complete description of such frames and also a construction of all such frames through the Kaczmarz algorithm.  We will introduce the key ideas of each of these topics then provide a more detailed description of the contents of the paper.

%\subsection{Operator Orbit Frames}
%\subsection{The Kaczmarz Algorithm}
%\subsection{Fourier Series Expansions}

\subsection{Operator Orbit Frames} Much is known about frames of the form $\{g_{n}\}_{n=0}^{\infty}=\{T^{n}g_{0}\}_{n=0}^{\infty}$ where $T$ is a bounded operator in a general Hilbert space. For example, Christensen, Hasannasab, and Philipp show that frames of this form are equivalent to operator orbit frames on a backward invariant subspace of the Hardy space in \cite{Christensen2020Frame}. Additionally, Aldroubi and Petrosyan prove many properties about the operators generating these frames such as they can't be unitary operators and normal operators never generate Riesz bases \cite{Aldroubi2017Dynamical}. Furthermore, some of the most robust frames in signal recovery are of the form $\{T^{n}g_{0}\}_{n=0}^{\infty}$, see \cite{Christensen2024Mystery} for details. We continue by providing the definition of a frame here:

\begin{defn}
Let $\{g_{n}\}_{n}\subseteq H$ where $H$ is a Hilbert space.
\begin{enumerate}
\item If there exists $A>0$ such that
$$A\|f\|^{2}\leq \sum_{n}|\langle f, g_{n}\rangle|^{2}$$
for all $f\in H$, then $\{g_{n}\}_{n}$ is called a \textbf{lower semi-frame}.
\item If there exists $B>0$ such that
$$\sum_{n}|\langle f, g_{n}\rangle|^{2}\leq B\|f\|^{2}$$
for all $f\in H$, then $\{g_{n}\}_{n}$ is called a \textbf{Bessel sequence}.
\item If there exists $A,B>0$ such that
    $$A\|f\|^{2}\leq \sum_{n}|\langle f, g_{n}\rangle|^{2}\leq B\|f\|^{2}$$
    for all $f\in H$, then $\{g_{n}\}_{n}$ is called a \textbf{frame}.
\item If for some $A>0$
$$A\|f\|^{2}=\sum_{n}|\langle f, g_{n}\rangle|^{2}$$
for all $f\in H$, then $\{g_{n}\}_{n}$ is called a \textbf{tight frame}.
\item If
$$\|f\|^{2}=\sum_{n}|\langle f, g_{n}\rangle|^{2}$$
for all $f\in H$, then $\{g_{n}\}_{n}$ is called a \textbf{Parseval frame}.
\item If $\{g_{n}\}$ is a frame, and
$$\sum_{n}c_{n}g_{n}=0\implies \{c_{n}\}=0,$$
whenever $\{c_{n}\}\in \ell^{2}$,
then $\{g_{n}\}_{n}$ is called a \textbf{Riesz basis}.
\end{enumerate}
Furthermore, if $\{g_{n}\}_{n}$ is a Bessel sequence in Hilbert space $H$,
the map on $H$: $f\to \{\langle f, g_{n}\rangle \}_{n}$ is called the \textbf{analysis operator}, its adjoint is called the \textbf{synthesis operator}, and 
we denote $S: H\to H$ as the positive operator such that
$$S(f)=\sum_{n}\langle f, g_{n}\rangle g_{n},$$ which is called the \textbf{frame operator}.
Additionally, if $\{g_{n}\}_{n}$ is a frame, then its associated frame operator is invertible, and $\{S^{-1}g_{n}\}_{n}$ is a frame such that
$$f=\sum_{n}\langle f, S^{-1}g_{n}\rangle g_{n}=\sum_{n}\langle f, g_{n}\rangle S^{-1}g_{n}$$ for all $f\in H$.
Furthermore, the convergence above is unconditional.
\end{defn}

\subsection{The Kaczmarz Algorithm} Our treatment of frames arising from operator orbits is motivated by the Kaczmarz algorithm in Hilbert spaces, see \cite{Haller2005Kaczmarz}.
Initially however, Stefan Kaczmarz introduced an iterative process of solving linear systems, which we now call the Kaczmarz algorithm \cite{Kaczmarz1993Approximate}. Although more recently, there is interest in asking when all vectors in a Hilbert space can be reconstructed from a sequence in a stable way:

\begin{defn}
For a Hilbert space $H$ and complete sequence of unit vectors $\{e_{n}\}_{n=0}^{\infty}\subseteq H$, define the \textbf{auxiliary sequence} of $\{e_{n}\}$ in $H$ recursively as follows:
$$g_{0}=e_{0}$$
$$g_{n}=e_{n}-\sum_{k=0}^{n-1}\langle e_{n},e_{k}\rangle g_{k} \ \forall n\geq 1.$$

If for all $x\in H$,
$$\sum_{k=0}^{n}\langle x,g_{k}\rangle e_{k}\to x,$$
then $\{e_{n}\}_{n=0}^{\infty}$ is called \textbf{effective}.
\end{defn}
Additionally, it is shown by Haller and Szwarc in \cite{Haller2005Kaczmarz} that $\{e_{n}\}_{n=0}^{\infty}$ is effective if and only if the auxiliary sequence $\{g_{n}\}_{n=0}^{\infty}$ is a Parseval frame. Due to Czaja and Tanis, it is also known that if $\{e_{n}\}_{n=0}^{\infty}$ is an effective Riesz basis, then $\{e_{n}\}_{n=0}^{\infty}$ is an orthonormal basis \cite{Czaja2013Kaczmarz}.

Furthermore, there is a dual Kaczmarz algorithm defined by Aboud, Curl, Harding, Vaughan, and Weber in \cite{Aboud2020Dual} that is a generalization of the Kaczmarz algorithm in Hilbert spaces:
\begin{defn}
Let $\{\phi_{n}\}_{n=0}^{\infty}$ and $\{\psi_{n}\}_{n=0}^{\infty}$ be complete sequences in a Hilbert space $H$ such that $$\langle \phi_{n},\psi_{n}\rangle=1$$ for all $n$. The \textbf{auxiliary sequence} for pair $(\phi_{n}, \psi_{n})$ is defined as follows:
$$g_{0}=\phi_{0},$$
for $n\geq 1$, $$g_{n}=\phi_{n}-\sum_{k=0}^{n-1}\langle \phi_{n},\psi_{k}\rangle g_{k}.$$

Furthermore, if the sequence $$x_{n}=\sum_{k=0}^{n}\langle x,g_{k}\rangle \psi_{k}\to x$$ for all $x\in H$, then the pair $(\phi_{n}, \psi_{n})$ is called \textbf{effective}.
\end{defn}

\subsection{Duality in Fourier Series Expansions} Because an effective Kaczmarz algorithm gives rise to a frame-like expansion, we make the definitions of these types of expansions clear:

\begin{defn}
In Hilbert space $H$, we say $\{x_{n}\}_{n=0}^{\infty}$ is \textbf{dextrodual} to $\{e_{n}\}_{n=0}^{\infty}$ if 
$$\sum_{n=0}^{k}\langle x,x_{n}\rangle e_{n}\to x$$
for all $x\in H$.

Similarly, we say that $\{x_{n}\}_{n\in \mathbb{Z}}$ is \textbf{Poisson dextrodual} to $\{e_{n}\}_{n\in \mathbb{Z}}$ if
$$\sum_{n=-k}^{k}\langle x,x_{n}\rangle e_{n}\to x$$
for all $x\in H$.
\end{defn}
Note, these definitions do not assume anything about unconditional convergence as in the case of frame series. Similar definitions of sequences with frame-like properties are also discussed by Li and Ogawa in \cite{Li2001Pseudo-duals}.

Noting the orthogonality of $\{e^{2\pi i nx}\}_{n\in \mathbb{Z}}$ in $L^{2}([0,1))$, there is recent interest in finding other measures $\mu$ that possess a sequence of exponentials that form an orthonormal basis for $L^{2}(\mu)$. An example of a singular measure that possesses this property was discovered by Jorgensen and Pedersen in \cite{Jorgensen1998Dense}, and further analysis of all orthonormal bases of exponential functions on this measure is discussed by Dutkay, Han, and Sun in \cite{Dutkay2009Spectra}. Moreover, a relaxation of this property for measures is also vastly studied, that is, finding measures $\mu$ that possess a sequence of exponentials that form a frame for $L^{2}(\mu)$. It was shown by Lai and Wang that there is a singular measure that does not admit an orthonormal basis of exponential functions but does admit a frame of exponential functions \cite{Lai2017Non-spectral}. Furthermore, it is known that if $\mu$ possesses a frame of exponential functions, then $\mu$ is of pure type, that is singular or absolutely continuous \cite{He2013Exponential}. In the case that $\mu$ is absolutely continuous, it is know exactly when $\mu$ possesses a frame of exponential functions by work of Lai in \cite{Lai2011Fourier}. While there are other known measures $\mu$ that are singular which possess a frame of exponential functions such as in \cite{Picioroaga2017Fourier}, when exactly singular measures have this property is still an open problem.
For more discussion on singular measures $\mu$ that possess a sequence of exponential functions that form a frame for $L^{2}(\mu)$, one can also read the work of Dutkay, Han, Sun, and Weber \cite{Dutkay2011Beurling}.

Furthermore, even though it is unknown exactly which singular measures possess a frame of exponential functions, it was shown in \cite{Herr2017Fourier} that if $\mu$ is a singular Borel probability measure on $[0,1)$, then $\mu$ possess a Parseval frame $\{g_{n}\}_{n=0}^{\infty}$ that is dextrodual to $\{e^{2\pi i nx}\}_{n=0}^{\infty}$, even though $\{e^{2\pi i nx}\}_{n=0}^{\infty}$ is never a frame for any singular measure. Therefore, this still allows $\mu$ to possess ``frame-like'' Fourier expansions. These ``frame-like'' Fourier expansions results are also extended to higher dimensions in \cite{Herr2022Fourier}, \cite{Herr2020Harmonic}, and \cite{Berner2024Fourier}. They are also discussed for measures not of pure type in \cite{Berner2024Frame}.

The primary result of this paper is that all frames generated by not surjective operator orbits arise from the Kaczmarz algorithm, and that these operators are similar to rank one perturbations of unitary operators (Theorem \ref{kaczmarzclass} and Corollary \ref{Tform}).  The other main goal of this paper is to classify all finite Borel measures $\mu$ on $[0,1)$ that possess sequences $\{T^{n}g_{0}\}_{n\in \mathbb{Z}}\subseteq L^{2}(\mu)$ where $T$ is a bounded invertible operator and $g_{0}\in L^{2}(\mu)$ that are Poisson dextrodual to $\{e^{2\pi i nx}\}_{n\in \mathbb{Z}}$.

However, we did not achieve a full classification of these measures, but we did classify measures of the specific case when $\{T^{n}g_{0}\}_{n\in \mathbb{Z}}$ is a frame (Corollary \ref{classcor}). Additionally, we were also able to describe a class of measures where $\{T^{n}g_{0}\}_{n\in \mathbb{Z}}$ is Poisson dextrodual to $\{e^{2\pi i nx}\}_{n\in \mathbb{Z}}$, and $\{T^{n}g_{0}\}_{n\in \mathbb{Z}}$ is not a frame (Corollary \ref{suffcase}).

The outline of this paper is as follows: In section \hyperref[S2]{2}, we show that every frame arising from not surjective operator orbits is the auxiliary sequence of an effective pair from the dual Kaczmarz algorithm and also the auxiliary sequence of an effective stationary sequence from the Kaczmarz algorithm in an equivalent Hilbert space. We also show that a not surjective operator generates an over-complete frame only if it is similar to a specific rank one perturbation of a unitary operator. In section \hyperref[S3]{3}, we describe a large class of frames arising from operator orbits that are dextrodual to the exponential functions in the case of singular measures on $[0,1)$. Then in section \hyperref[S4]{4}, we provide necessity conditions for Borel measures on $[0,1)$ to possess a two-sided operator orbit that is dextrodual to the exponential functions. Furthermore, in section \hyperref[S5]{5}, we show a special case when these conditions are sufficient. Moreover in section \hyperref[S6]{6}, we classify all finite Borel measures on $[0,1)$ that possess a two-sided operator orbit frame that is dextrodual to the exponential functions. Finally in section \hyperref[S7]{7}, we provide a couple of classification results for measures possessing certain types of frame-like Fourier expansions.  

Throughout this paper, we denote $M_{g}$ as the multiplication operator on an $L^{2}$ space by element $g$, we denote $\mathbb{1}$ as the constant $1$ function in an $L^{2}$ space, and we denote $\langle \cdot , \cdot \rangle_{H}$ and $\langle \cdot , \cdot \rangle_{v}$ as inner products in Hilbert space $H$ and $L^{2}(v)$ respectively. We will also assume our Hilbert spaces are infinite dimensional.

\section{Operator orbits and the Kaczmarz algorithm}\label{S2}
In this section, we show every frame of the form $\{T^{n}g_{0}\}_{n=0}^{\infty}$ in Hilbert space $H$ arises from an effective dual Kaczmarz algorithm in $H$ when $T$ is not surjective, and it is dextrodual to an operator orbit, which is not a Bessel sequence. We also show that $\{T^{n}g_{0}\}_{n=0}^{\infty}$ arises from an effective Kaczmarz algorithm in an equivalent Hilbert space. Finally, in this section we show that not surjective operators generating over-complete frames are rank one perturbations of unitary operators up to similarity. For more discussion about rank one perturbations of unitary operators, see \cite{Clark1972One}. Additionally, it was shown by Szwarc in \cite{Szwarc2007Kaczmarz} that every sequence $\{g_{n}\}_{n}$ with Bessel bound at most 1 such that $g_{0}$ is a unit vector is obtained by the Kaczmarz algorithm. Our approach however, does not assume the first vector is a unit vector and looks at the operator orbit case.
To prove these results, we need some results in \cite{Berner2024Frame} and \cite{Herr2017Fourier}.

\begin{defn}
Recall the Hardy space on the disk:
$$H^{2}(\mathbb{D})=\{w\to \sum_{n=0}^{\infty}c_{n}w^{n}: \sum_{n=0}^{\infty}|c_{n}|^{2}<\infty\}.$$
Furthermore, an \textbf{inner function} $b(w)$ is a bounded analytic function on $\mathbb{D}$ such that 
$$\|b(w)f(w)\|_{H^{2}(\mathbb{D})}=\|f(w)\|_{H^{2}(\mathbb{D})}$$ for any $f\in H^{2}(\mathbb{D})$.
\end{defn}
The following result is inspired by Beurling in \cite{Beurling1948Two} where the author classified all subspaces of the Hardy space that are invariant under the forward shift:

\begin{prop}[Proposition 5.3 \cite{Berner2024Frame}]\label{backwardop}
If $\{g_{n}\}_{n=0}^{\infty}$ is a frame in Hilbert space $H$ and not a Riesz basis, then $$\{\sum_{n=0}^{\infty}\langle f, g_{n}\rangle z^{n}: f\in H\}=[b(z)H^{2}(\mathbb{D})]^{\perp}$$ for some inner function $b(z)$ if and only if there is a bounded operator that maps $g_{n}$ to $g_{n+1}$ for all $n$.
\end{prop}

\begin{cor}[Corollary 5.1 \cite{Berner2024Frame}]\label{eqcor}
Suppose $\{T^{n}g_{0}\}_{n=0}^{\infty}=\{g_{n}\}$ is a frame in Hilbert space $H$ and not a Riesz basis where $T\in B(H)$, $g_{0}\in H$, and $\{T^{n}g_{0}\}_{n=1}^{\infty}$ is not complete. Then there is an auxiliary sequence $\{h_{n}\}$ of $\{e^{2\pi i nx}\}_{n=0}^{\infty}$ in $L^{2}(v)$ where $v$ is a singular Borel probability measure on $[0,1)$ such that
$$\langle S^{-1}g_{n},g_{k}\rangle = \langle h_{n},h_{k}\rangle $$ for all $n$ and $k$.
\end{cor}

The following result by Christensen and Hasannasab in \cite{Christensen2019Frame}, lets us apply Corollary \ref{eqcor} to the case when $T$ is not surjective:

\begin{prop}[Proposition 2.5 \cite{Christensen2019Frame}]
If $\{T^{n}g_{0}\}_{n=0}^{\infty}=\{g_{n}\}$ is a frame in Hilbert space $H$, then the range of $T$ is $\overline{span\{T^{n}g_{0}\}_{n=1}^{\infty}}$.
\end{prop}

Additionally, the following result was inspired by the work of Kwapien and Mycielski in \cite{Kwapien2001Kaczmarz} where the authors classified all effective stationary sequences:

\begin{thm}[Theorem 1, Proposition 1 \cite{Herr2017Fourier}]\label{singular}
    If $v$ is a singular Borel probability measure on $[0,1)$, then $\{e^{2\pi i nx}\}_{n=0}^{\infty}$ is effective in $L^{2}(v)$.
    Furthermore, let $\{g_{n}\}$ be the auxiliary sequence of $\{e^{2\pi i nx}\}_{n=0}^{\infty}$. Then
    $$\{\sum_{n=0}^{\infty}\langle f, g_{n}\rangle z^{n}: f\in L^{2}(\mu)\}=[b(z)H^{2}(\mathbb{D})]^{\perp}$$ for some inner function $b(z)$.
\end{thm}

\begin{thm}\label{genbackward}
Let $\{T^{n}g_{0}\}_{n=0}^{\infty}=\{g_{n}\}$ be a frame and not a Riesz basis in Hilbert space $H$ where $T\in B(H)$, $T$ is not surjective, and $g_{0}\in H$. Then there is a singular Borel probability measure $v$ on $[0,1)$ and unitary $U\in B(H, L^{2}(v))$ such that $\{T^{n}g_{0}\}_{n=0}^{\infty}$ is the auxiliary sequence of effective pair 
$$(S^{\frac{1}{2}}U^{*}(e^{2\pi i nx}),S^{-\frac{1}{2}}U^{*}(e^{2\pi i nx}))$$ where $S$ denotes the frame operator associated with $\{T^{n}g_{0}\}_{n=0}^{\infty}$.
\end{thm}
\begin{proof}
By Corollary \ref{eqcor}, we see that there is a singular Borel probability measure $v$ on $[0,1)$ and an isometry $U\in B(H, L^{2}(v))$ such that $$U(S^{-\frac{1}{2}}g_{n})=h_{n}$$ for all $n$ where $\{h_{n}\}$ is the auxiliary sequence of $\{e^{2\pi i nx}\}_{n=0}^{\infty}$ in $L^{2}(v)$. Note that $U$ is a unitary since $\{e^{2\pi i nx}\}_{n=0}^{\infty}$ is effective in $L^{2}(v)$. Let $\{g_{n}'\}_{n=0}^{\infty}$ be the auxiliary sequence of the pair $(S^{\frac{1}{2}}U^{*}(e^{2\pi i nx}),S^{-\frac{1}{2}}U^{*}(e^{2\pi i nx}))$.

First note that $$g_{0}'=S^{\frac{1}{2}}U^{*}(\mathbb{1})=g_{0}.$$

Now suppose that $g_{k}'=g_{k}$ for all $0\leq k\leq n-1$.
We have
\begin{equation}\label{auxeq}
g_{n}'=S^{\frac{1}{2}}U^{*}(e^{2\pi i nx})-\sum_{k=0}^{n-1}\langle S^{\frac{1}{2}}U^{*}(e^{2\pi i nx}), S^{-\frac{1}{2}}U^{*}(e^{2\pi i kx})\rangle_{H} g_{k}.
\end{equation}
Now by the recursive definition of the $\{h_{n}\}$, we have
\begin{equation}\label{recureq}
\begin{split}
S^{\frac{1}{2}}U^{*}(e^{2\pi i nx})&=S^{\frac{1}{2}}U^{*}(h_{n}+\sum_{k=0}^{n-1}\langle e^{2\pi i nx}, e^{2\pi i kx}\rangle_{v}h_{k})\\
&=g_{n}+\sum_{k=0}^{n-1}\langle e^{2\pi i nx}, e^{2\pi i kx}\rangle_{v}g_{k}.\\  
\end{split}
\end{equation}
Now because 
$$\langle S^{\frac{1}{2}}U^{*}(e^{2\pi i nx}), S^{-\frac{1}{2}}U^{*}(e^{2\pi i kx})\rangle_{H}=\langle e^{2\pi i nx}, e^{2\pi i kx}\rangle_{v}$$ for all $n,k$, we have
$$g_{n}=g_{n}'$$ for all $n$ by substituting equation (\ref{recureq}) into equation (\ref{auxeq}).

Now to show $(S^{\frac{1}{2}}U^{*}(e^{2\pi i nx}),S^{-\frac{1}{2}}U^{*}(e^{2\pi i nx}))$ is an effective pair, let $x\in H$ and for each $n$, let $$x_{n}=\sum_{k=0}^{n}\langle x, g_{k}'\rangle S^{-\frac{1}{2}}U^{*}(e^{2\pi i kx})=\sum_{k=0}^{n}\langle x, g_{k}\rangle S^{-\frac{1}{2}}U^{*}(e^{2\pi i kx}).$$ Define for each $n$, $$y_{n}=\sum_{k=0}^{n}\langle US^{\frac{1}{2}}x, h_{k}\rangle e^{2\pi i kx}=\sum_{k=0}^{n}\langle x, g_{k}\rangle e^{2\pi i kx}.$$
Now by Theorem \ref{singular}, we have for all $x\in H$,
$$
y_{n}\to US^{\frac{1}{2}}x \implies S^{-\frac{1}{2}}U^{*}y_{n}=x_{n}\to x.$$
Therefore, $(S^{\frac{1}{2}}U^{*}(e^{2\pi i nx}),S^{-\frac{1}{2}}U^{*}(e^{2\pi i nx}))$ is an effective pair.
\end{proof}
\begin{rem}
Note that if $\{g_{n}\}_{n=0}^{\infty}$ is a Riesz basis in Theorem \ref{genbackward} with frame operator $S$, it is easy to see that $\{g_{n}\}_{n=0}^{\infty}$ is the auxiliary sequence of the effective pair $(g_{n}, S^{-1}g_{n})$. This follows because $\{S^{-1}g_{n}\}$ is the biorthogonal sequence of $\{g_{n}\}_{n=0}^{\infty}$ and because $\{g_{n}\}_{n=0}^{\infty}$ is dextrodual to $\{S^{-1}g_{n}\}$.
\end{rem}

As a result, every over-complete frame of the form 
$\{T^{n}g_{0}\}_{n=0}^{\infty}$ when $T$ is not surjective is dextrodual to an orbit sequence which is not a frame:
\begin{cor}\label{framelikedualcor}
Let $\{T^{n}g_{0}\}_{n=0}^{\infty}=\{g_{n}\}$ be a frame and not a Riesz basis in Hilbert space $H$ where $T\in B(H)$, $T$ is not surjective, and $g_{0}\in H$. Define $U$ and $S$ from Theorem \ref{genbackward}. Then $\{g_{n}\}_{n=0}^{\infty}$ is dextrodual to 
$$\{\mathcal{M}^{n}(S^{-\frac{1}{2}}U^{*}(\mathbb{1}))\}_{n=0}^{\infty}$$ in $H$, which is not a Bessel sequence where $\mathcal{M}=S^{-\frac{1}{2}}U^{*}M_{e^{2\pi i x}}US^{\frac{1}{2}}$.

\end{cor}
\begin{proof}
We have already shown in Theorem \ref{genbackward}, that $\{g_{n}\}_{n=0}^{\infty}$ is dextrodual to 
$\{\mathcal{M}^{n}(S^{-\frac{1}{2}}U^{*}(\mathbb{1}))\}_{n=0}^{\infty}$ in $H$.

Now assume for the sake of contradiction that $\{\mathcal{M}^{n}(S^{-\frac{1}{2}}U^{*}(\mathbb{1}))\}_{n=0}^{\infty}$ is a Bessel sequence. 
Then we have in particular,
$$\sum_{n=0}^{\infty}|\langle S^{\frac{1}{2}}U^{*}(\mathbb{1}),S^{-\frac{1}{2}}U^{*}(e^{2\pi i nx})\rangle_{H}|^{2}<\infty.$$
But this implies that the Fourier coefficients of $v$ are square summable, which means that $v$ has the same Fourier coefficients as an absolutely continuous complex measure. Then since $v$ is singular, it is the zero measure, which is a contradiction.
\end{proof}

Theorem \ref{genbackward} will also give us that when $T$ is not surjective, $\{T^{n}g_{0}\}_{n=0}^{\infty}$ is an over-complete frame only if $T$ is a rank one perturbation of multiplication by the exponential up to similarity.

\begin{cor}\label{Tform}
Let $\{T^{n}g_{0}\}_{n=0}^{\infty}=\{g_{n}\}$ be a frame in Hilbert space $H$ and not a Riesz basis where $T\in B(H)$, $T$ is not surjective, and $g_{0}\in H$. Define $U$ and $S$ from Theorem \ref{genbackward}. Then for any $f\in H$,
$$Tf=(S^{\frac{1}{2}}U^{*}M_{e^{2\pi i x}}US^{-\frac{1}{2}})f-\langle f, S^{-\frac{1}{2}}U^{*}(e^{-2\pi ix})\rangle g_{0}.$$

Conversely, for any Hilbert space $H$, any singular Borel probability measure $v$ on $[0,1)$, and any $V\in B(H, L^{2}(v))$ that is invertible, define $T\in B(H)$ where for $f\in H$,
$$Tf=V^{-1}M_{e^{2\pi i x}}Vf-\langle f, V^{*}(e^{-2\pi ix})\rangle V^{-1}(\mathbb{1}).$$
Then $\{T^{n}(V^{-1}(\mathbb{1}))\}_{n=0}^{\infty}$ is a frame and not a Riesz basis in $H$.
\end{cor}
\begin{proof}
Let $\mathcal{M}=S^{-\frac{1}{2}}U^{*}M_{e^{2\pi i x}}US^{\frac{1}{2}}$. By Corollary \ref{framelikedualcor}, we have
$$\mathcal{M}T^{*}f=\sum_{n=0}^{\infty}\langle f, T^{n+1}g_{0}\rangle \mathcal{M}^{n+1}(S^{-\frac{1}{2}}U^{*}(\mathbb{1}))= f-\langle f, g_{0}\rangle S^{-\frac{1}{2}}U^{*}(\mathbb{1})$$ for all $f\in H$. By adjoint calculations, we get
$$Tf=(\mathcal{M}^{*})^{-1}f-\langle f, \mathcal{M}^{-1}S^{-\frac{1}{2}}U^{*}(\mathbb{1})\rangle g_{0}.$$

Conversely, let $\{h_{n}\}$ be the auxiliary sequence of $\{e^{2\pi i nx}\}_{n=0}^{\infty}$ in $L^{2}(v)$. As a result of Theorem \ref{singular}, there is an $L\in B(L^{2}(v))$ such that $$h_{n}=L^{n}(\mathbb{1})$$ for all $n$. By a similar argument due to the effectiveness of $\{e^{2\pi i nx}\}_{n=0}^{\infty}$, we get can get that
$$Lf=e^{2\pi i x}f-\langle f, e^{-2\pi ix}\rangle$$ for any $f\in L^{2}(v)$. Now note that
$$V^{-1}LV=T\implies T^{n}(V^{-1}(\mathbb{1}))=V^{-1}(h_{n})$$ for all $n$. Therefore since $\{h_{n}\}$ is a Parseval frame, for all $f\in H$,
$$\sum_{n=0}^{\infty}|\langle f, T^{n}(V^{-1}(\mathbb{1}))\rangle |^{2}=\sum_{n=0}^{\infty}|\langle f, V^{-1}(h_{n})\rangle |^{2}=\|(V^{-1})^{*}(f)\|^{2}, $$
showing that $\{T^{n}(V^{-1}(\mathbb{1}))\}_{n=0}^{\infty}$ is a frame in $H$. 
Finally, note that the analysis operator associated with $\{h_{n}\}$ does not have dense range as a result of Theorem \ref{singular}, making its synthesis operator not injective. Therefore, $\{T^{n}(V^{-1}(\mathbb{1}))\}_{n=0}^{\infty}$ is not a Riesz basis since it is similar to a frame with not injective synthesis operator.
\end{proof}

Furthermore, we can show that frames of the form $\{T^{n}g_{0}\}_{n=0}^{\infty}$ when $T$ is not surjective arise from an effective Kaczmarz algorithm in an equivalent Hilbert space:

\begin{thm}\label{kaczmarzclass}
Let $\{T^{n}g_{0}\}_{n=0}^{\infty}=\{g_{n}\}$ be a frame and not a Riesz basis in Hilbert space $H$ where $T$ is not surjective. Define $U$ and $S$ from Theorem \ref{genbackward}, and $H'$ as the Hilbert space that is a renorming of $H$ where for $f, g\in H$
$$\langle f, g\rangle_{H'}=\langle S^{-\frac{1}{2}}f,S^{-\frac{1}{2}}g\rangle_{H}.$$
Then $\{g_{n}\}$ is the auxiliary sequence of effective sequence 
$$\{(S^{\frac{1}{2}}U^{*}M_{e^{2\pi i x}}US^{-\frac{1}{2}})^{n}(g_{0})\}_{n=0}^{\infty}$$ in $H'$.
Furthermore, $S^{\frac{1}{2}}U^{*}M_{e^{2\pi i x}}US^{-\frac{1}{2}}$ is a unitary on $H$ if and only if
$$S=U^{*}M_{g}U$$ for some $g\in L^{\infty}(v)$.

\end{thm}
\begin{proof}
Showing that $\{g_{n}\}$ is the auxiliary sequence of $$\{(S^{\frac{1}{2}}U^{*}M_{e^{2\pi i x}}US^{-\frac{1}{2}})^{n}(g_{0})\}_{n=0}^{\infty}$$ in $H'$ is similar to the argument in Theorem \ref{genbackward}.

It is known that $\{S^{-\frac{1}{2}}g_{n}\}$ is a Parseval frame in $H$; therefore, $\{g_{n}\}$ is a Parseval frame in $H'$. It follows that $\{(S^{\frac{1}{2}}U^{*}M_{e^{2\pi i x}}US^{-\frac{1}{2}})^{n}(g_{0})\}_{n=0}^{\infty}$ is effective in $H'$.

Now suppose that $$S=U^{*}M_{g}U$$ for some $g\in L^{\infty}(v)$. Then 
$$S^{\frac{1}{2}}=U^{*}M_{\sqrt{g}}U,$$ and it is easy to see that
$$S^{\frac{1}{2}}U^{*}M_{e^{2\pi i x}}US^{-\frac{1}{2}}=U^{*}M_{e^{2\pi i x}}U.$$

Finally, suppose that $S^{\frac{1}{2}}U^{*}M_{e^{2\pi i x}}US^{-\frac{1}{2}}$ is unitary.
Then by an adjoint and inverse calculation, we have
$$S^{-\frac{1}{2}}U^{*}M_{e^{-2\pi i x}}US^{\frac{1}{2}}=S^{\frac{1}{2}}U^{*}M_{e^{-2\pi i x}}US^{-\frac{1}{2}}.$$
Therefore, $S$ commutes with $U^{*}M_{e^{-2\pi i x}}U$. It follows that $USU^{*}$ commutes with $M_{e^{-2\pi i x}}$ so that
$USU^{*}$ is a multiplication operator on $L^{2}(v)$.
\end{proof}

\begin{rem}
If $\{g_{n}\}$ is a Riesz basis in Theorem \ref{kaczmarzclass} with frame operator $S$, then $\{g_{n}\}$ is an orthonormal basis in $H'$. It is clear that an orthonormal basis is the auxiliary sequence of itself.
\end{rem}

\section{Fourier series for singular measures and operator orbits}\label{S3}

To study frames of the form $\{T^{n}g_{0}\}_{n=0}^{\infty}$ that are dextrodual to 
$\{e^{2\pi i nx}\}_{n=0}^{\infty}$, initially one might try to answer the following question:  Can we classify finite Borel measures $\mu$ on $[0,1)$ such that there is a frame of the form $\{T^{n}g_{0}\}_{n=0}^{\infty}$ that is dextrodual to $\{e^{2\pi i nx}\}$?
But by the F. and M. Riesz theorem, a measure with this property must be singular, and by Theorem \ref{singular}, we know such expansions exist for all singular measures.
Instead in this section, we study the question: Given a finite singular Borel measure on $[0,1)$, can we describe all of the frames of the form $\{T^{n}g_{0}\}_{n=0}^{\infty}$ that are dextrodual to $\{e^{2\pi i nx}\}$ when $T$ is not surjective?
Our first result is that these types of Fourier expansions for singular measures are not unique, which is also discussed in \cite{Herr2019Positive}.

\begin{prop}\label{exist}
Let $\mu$ be a finite singular Borel measure on $[0,1)$ and $g_{0}\in L^{2}(\mu)$ such that
\begin{enumerate}
    \item $g_{0}$ is non-negative,
    \item $\langle g_{0},1\rangle=1$,
    \item $g_{0},\frac{1}{g_{0}}\in L^{\infty}(\mu)$.
\end{enumerate}
Let $g_{n}=g_{0}h_{n}$ for all $n$ where $\{h_{n}\}$ is the auxiliary sequence of $\{e^{2\pi i nx}\}_{n=0}^{\infty}$ in $L^{2}(g_{0}\mu)$. Then
$\{g_{n}\}$ is a frame of the form $\{T^{n}g_{0}\}_{n=0}^{\infty}$ for some operator $T$ that is dextrodual to $\{e^{2\pi i nx}\}_{n=0}^{\infty}$ in $L^{2}(\mu).$
\end{prop}
\begin{proof}
By Theorem \ref{singular},
for all $f\in L^{2}(g_{0}\mu)$,
$$f=\sum_{n=0}^{\infty}\langle f,h_{n}\rangle_{g_{0}\mu}e^{2\pi i nx} =\sum_{n=0}^{\infty}\langle f, g_{n}\rangle_{\mu}e^{2\pi i nx}.$$

Since $L^{2}(g_{0}\mu)$ convergence is equivalent to $L^{2}(\mu)$ convergence from $M_{g_{0}}$ being invertible in $L^{2}(\mu)$, $\{g_{n}\}$ is dextrodual to $\{e^{2\pi i nx}\}$ in $L^{2}(\mu).$ Furthermore, $\{g_{n}\}$ is also equivalent to an auxiliary sequence, so by Proposition \ref{backwardop}, $\{g_{n}\}$ is of the desired form.
\end{proof}
However, we can explicitly state the form of frames $\{T^{n}g_{0}\}_{n=0}^{\infty}$ that are dextrodual to $\{e^{2\pi inx}\}$.
We also show that if a frame arising from a not surjective operator orbit is dextrodual to the exponential functions, then it arises from another effective dual Kaczmarz algorithm. 
To show this, we draw on another result in \cite{Berner2024Frame}.

\begin{lem}[Lemma 2.3 \cite{Berner2024Frame}]\label{RBL}
In Hilbert space $H$, if Riesz basis $\{g_{n}\}_{n=0}^{\infty}$ is dextrodual to $\{h_{n}\}_{n=0}^{\infty}$, then $\{h_{n}\}_{n=0}^{\infty}$ is a Riesz basis.
\end{lem}

\begin{thm}\label{Mainsingular}
Let $\{T^{n}g_{0}\}_{n=0}^{\infty}=\{g_{n}\}$ be a frame that is dextrodual to $\{e^{2\pi i nx}\}_{n=0}^{\infty}$ in $L^{2}(\mu)$ where $T\in B(H)$, $T$ is not surjective, and $\mu$ is a finite singular Borel measure on $[0,1)$. Denote the frame operator of $\{g_{n}\}$ by $S$. Then
\begin{enumerate}
    \item For $f\in L^{2}(\mu)$, $Tf=e^{2\pi i x}f-\langle e^{2\pi i x}f,\mathbb{1}\rangle_{\mu} g_{0}$.
    \item $S(\mathbb{1})=g_{0}$.
    \item $\{g_{n}\}$ is the auxiliary sequence of the effective pair $(S(\mathbb{1})e^{2\pi i nx},e^{2\pi i nx})$ in $L^{2}(\mu)$.
\end{enumerate}

\end{thm}
\begin{proof}
$(1)$

Similar to the calculation in Corollary \ref{Tform},
we have for $f \in L^{2}(\mu)$,
$$e^{2\pi i x}T^{*}f=\sum_{n=0}^{\infty}\langle f ,T^{n+1}g_{0}\rangle e^{2\pi i (n+1)x} =f-\langle f, g_{0}\rangle .$$
From here, a simple calculation shows $(1)$.

$(2)$

First we show $\{T^{n}g_{0}\}=\{g_{n}\}$ is not a Riesz basis. Suppose the contrary. By Lemma \ref{RBL}, we have $\{e^{2\pi i nx}\}_{n=0}^{\infty}$ is a Riesz basis in $L^{2}(\mu)$, but this is a contradiction since Proposition \ref{exist} shows Fourier expansions are not unique of this type.

By Corollary \ref{eqcor}, there is an auxiliary sequence $\{h_{n}\}$ of a singular Borel probability measure $v$ on $[0,1)$ such that
$$\langle S^{-1}g_{n},g_{k}\rangle_{\mu}=\langle h_{n},h_{k}\rangle_{v}$$ for all $n$ and $k$.

We have
$$\langle S^{-1}g_{0},Tg_{n-1}\rangle=\langle S^{-1}g_{0}, e^{2\pi i x}g_{n-1}\rangle-\langle \mathbb{1}, e^{2\pi i x}g_{n-1}\rangle \langle S^{-1}g_{0},g_{0}\rangle=0$$ for all $n\geq 1$ since $h_{0}$ is orthogonal to $h_{n}$ for all $n\geq 1$.
Furthermore,
$$\langle S^{-1}g_{0},g_{0}\rangle=1$$ since $h_{0}$ is a unit vector.
Now since $\{e^{2\pi i x}g_{n-1}\}_{n=1}^{\infty}$ is dense in $L^{2}(\mu)$, it must be that
$$S^{-1}g_{0}=\mathbb{1}.$$

$(3)$

Note that
$$\langle S(\mathbb{1}),\mathbb{1}\rangle=\langle S^{-\frac{1}{2}}S(\mathbb{1}),S^{-\frac{1}{2}}S(\mathbb{1})\rangle=\langle S^{-1}g_{0},g_{0}\rangle=1.$$

Let $\{g_{n}'\}$ be the auxiliary sequence of pair $(g_{0}e^{2\pi i nx},e^{2\pi i nx})$.
It is clear $g_{0}=g_{0}'$.

Suppose $g_{k}=g_{k}'$ for all $1\leq k\leq n$.

We have
$$g_{n}'=g_{0}e^{2\pi i nx}-\sum_{k=0}^{n-1}\langle g_{0}e^{2\pi i (n-k)x},\mathbb{1}\rangle g_{k}\implies$$
\begin{equation}
\begin{split}
Tg_{n}'&=g_{n+1}\\
&=T(g_{0}e^{2\pi i nx})-\sum_{k=0}^{n-1}\langle g_{0}e^{2\pi i (n-k)x},\mathbb{1}\rangle g_{k+1}\\
&=g_{0}e^{2\pi i (n+1)x}-\langle g_{0}e^{2\pi i (n+1)x}, \mathbb{1}\rangle g_{0}-\sum_{k=0}^{n-1}\langle g_{0}e^{2\pi i (n-k)x},\mathbb{1}\rangle g_{k+1}\\
&=g_{n+1}'.
\end{split}
\end{equation}

\end{proof}
Now we can establish a uniqueness result about these types of Fourier expansions:
\begin{cor}
Suppose that $\mu$ is a finite singular Borel measure on $[0,1)$.
\begin{enumerate}
    \item There is exactly one frame of the form $\{T^{n}(\frac{\mathbb{1}}{\mu([0,1))})\}_{n=0}^{\infty}$ for some $T\in B(L^{2}(\mu))$ that is dextrodual to $\{e^{2\pi i nx}\}_{n=0}^{\infty}$.
    \item There is exactly one tight frame of the form $\{T^{n}g_{0}\}_{n=0}^{\infty}$ where $T\in B(L^{2}(\mu))$, $T$ is not surjective, and $g_{0}\in L^{2}(\mu)$ that is dextrodual to $\{e^{2\pi i nx}\}_{n=0}^{\infty}$.
\end{enumerate}
\end{cor}
\begin{proof}
$(1)$

Existence of this frame follows from Proposition \ref{exist}, and uniqueness of this frame follows from statement $(1)$ of Theorem \ref{Mainsingular}.

$(2)$

The unique frame of the form $\{T^{n}(\frac{\mathbb{1}}{\mu([0,1))})\}_{n=0}^{\infty}$ where $T\in B(L^{2}(\mu))$ that is dextrodual to $\{e^{2\pi i nx}\}_{n=0}^{\infty}$ is $\{\frac{\mathbb{1}}{\mu([0,1))} h_{n}\}$ where $\{h_{n}\}$ is the auxiliary sequence of $\{e^{2\pi i nx}\}$ in $L^{2}(\frac{\mathbb{1}}{\mu([0,1))} \mu)$ by Proposition \ref{exist}. Since $\{h_{n}\}$ is a Parseval frame in $L^{2}(\frac{\mathbb{1}}{\mu([0,1))} \mu)$ by Theorem \ref{singular}, it is easy to show that $\{\frac{\mathbb{1}}{\mu([0,1))} h_{n}\}$ is a tight frame in $L^{2}(\mu)$.

Now if $\{T^{n}g_{0}\}_{n=0}^{\infty}$ is a tight frame where $T\in B(L^{2}(\mu))$ and $g_{0}\in L^{2}(\mu)$ that is dextrodual to $\{e^{2\pi i nx}\}_{n=0}^{\infty}$, the frame operator of $\{T^{n}g_{0}\}_{n=0}^{\infty}$ is a scale of the identity operator, so by statement $(2)$ and $(3)$ of Theorem \ref{Mainsingular}, $g_{0}$ is a constant function with $\mu$ average $1$. Then uniqueness of this frame follows from statement $(1)$ of Theorem \ref{Mainsingular}.
\end{proof}

Our last result of this section is about frames of the form $\{T^{n}g_{0}\}$ from Theorem \ref{Mainsingular}. Must they be of the form in Proposition \ref{exist} or do others exist? We were unable to show if expansions exist outside of the form in Proposition \ref{exist}, but the following result says that if the frame operator of $\{T^{n}g_{0}\}$ is nice, then the expansions are of the form in Proposition \ref{exist}.
\begin{thm}
Assuming the hypothesis of Theorem \ref{Mainsingular}, let $U$ be defined from Theorem \ref{genbackward}. Then the following are equivalent:
\begin{enumerate}
\item $S(\mathbb{1})$ is non-negative
\item $S$ is multiplication by $g_{0}$.
\item $$M_{g_{0}}(M_{e^{-2\pi i x}}(\mathbb{1}))=S(M_{e^{-2\pi i x}}(\mathbb{1})).$$
\item $$U^{*}(M_{e^{-2\pi i x}}(\mathbb{1}_{v}))=S^{\frac{1}{2}}(M_{e^{-2\pi i x}}(\mathbb{1}_{\mu})).$$
\end{enumerate}
Moreover, if statement $(1)$ is satisfied, for all $n$, $$g_{n}=S(\mathbb{1})h_{n}$$ where $\{h_{n}\}$ be auxiliary sequence of $\{e^{2\pi i nx}\}_{n=0}^{\infty}$ in $L^{2}(S(\mathbb{1}) \mu)$,
and $S(\mathbb{1}), \frac{1}{S(\mathbb{1})}\in L^{\infty}(\mu)$.
\end{thm}
\begin{proof}
Suppose $(1)$.
By Theorem \ref{singular},
$$h_{n}=L^n\mathbb{1}_{v}$$ for all $n$ for some $L\in B(L^{2}(S(\mathbb{1}) \mu))$ where for all $f\in L^{2}(S(\mathbb{1}) \mu)$,
$$Lf=e^{2\pi i x}f-\langle e^{2\pi i x}f,1\rangle_{S(\mathbb{1})\mu}.$$
One can show $S(\mathbb{1})L^{n}1=T^{n}g_{0}$ for all $n$ by induction using Theorem \ref{Mainsingular}.

Let $B\subseteq [0,1)$ be Borel. We have by the frame condition of $\{g_{n}\}$ and Theorem \ref{singular},
$$\int_{B}S(\mathbb{1})d\mu=\|\chi_{B}\|_{S(\mathbb{1})\mu}^{2}=\sum_{n=0}^{\infty}|\langle \chi_{B},h_{n}\rangle_{S(\mathbb{1})\mu}|^{2}=\sum_{n=0}^{\infty}|\langle \chi_{B},g_{n}\rangle_{\mu}|^{2}\cong \mu(B).$$
Since the frame bounds from $\{g_{n}\}$ are independent of $B$, $S(\mathbb{1})$ is bounded above and below.

$(1) \implies (2)$

For any $f\in L^{2}(\mu)$,
$$S(f)=\sum_{n=0}^{\infty}\langle f, h_{n}\rangle_{S(\mathbb{1})\mu} S(\mathbb{1})h_{n}=S(\mathbb{1})f=g_{0}f$$
since $L^{2}(S(\mathbb{1}) \mu)$ and $L^{2}(\mu)$ convergence are equivalent and the frame operator associated with $\{h_{n}\}$ in $L^{2}(S(\mathbb{1}) \mu)$ is the identity operator.

$(2)\implies (1)$

Since $S$ is positive definite, for any $B\subseteq [0,1)$ Borel, 
$$\langle S(\chi_{B}), \chi_{B}\rangle =\int_{B}g_{0}\geq 0,$$
then $g_{0}$ must be non-negative.

$(2)\implies (3)$

This is obvious.

$(3)\implies (2)$

For $f\in L^{2}(\mu)$ we have,

$$TST^{*}(f)=\sum_{n=0}^{\infty}\langle f, T^{n+1}g_{0}\rangle T^{n+1}g_{0}=S(f)-\langle f,g_{0}\rangle g_{0} .$$

Now making a substitution for the operator $T$ from Theorem \ref{Mainsingular}, we have

\begin{equation}\label{frameopcalc}
e^{2\pi i x}S(e^{-2\pi i x}(f-\langle f,g_{0}\rangle) )+\langle S(e^{-2\pi i x}(\langle f,g_{0}\rangle -f)), e^{-2\pi i x}\rangle g_{0}
\end{equation}
$$=S(f)-\langle f,g_{0}\rangle g_{0} .$$

Now suppose that $f\in \{g_{0}\}^{\perp}$. Then in equation (\ref{frameopcalc}) we have
$$e^{2\pi i x}S(e^{-2\pi i x}f)-\langle S(e^{-2\pi ix}f),e^{-2\pi i x}\rangle g_{0}=S(f).$$
But recalling statement $(3)$ and that $S$ is self-adjoint, we have from equation (\ref{frameopcalc}),

\begin{equation}
\begin{split}
e^{2\pi i x}S(e^{-2\pi i x}f)-\langle e^{-2\pi ix}f,S(e^{-2\pi i x})\rangle g_{0}&=e^{2\pi i x}S(e^{-2\pi i x}f)-\langle e^{-2\pi ix}f,e^{-2\pi i x}g_{0}\rangle g_{0}\\
&=e^{2\pi i x}S(e^{-2\pi i x}f)\\
&=S(f).
\end{split}
\end{equation}

Now suppose $f=g_{0}$ in equation (\ref{frameopcalc}), we have
$$e^{2\pi ix}S(e^{-2\pi i x}g_{0})-e^{2\pi ix}\|g_{0}\|^{2}S(e^{-2\pi i x})+\langle e^{-2\pi i x}(\|g_{0}\|^{2}-g_{0}),S(e^{-2\pi i x})\rangle g_{0} $$
$$=S(g_{0})-\|g_{0}\|^{2}g_{0}.$$
By statement $(3)$ and recalling that $\langle \mathbb{1},g_{0}\rangle=1$,

$$e^{2\pi ix}S(e^{-2\pi i x}g_{0})-\|g_{0}\|^{2}g_{0}=S(g_{0})-\|g_{0}\|^{2}g_{0}.$$
Therefore, $S$ commutes with multiplication by $e^{2\pi i x}$, and it must be that
$S$ is multiplication by $S(\mathbb{1})=g_{0}$.

$(4)\implies (2)$

By looking at the forms of $T$ from Corollary \ref{Tform} and Theorem \ref{Mainsingular} and since we have
$$S^{-\frac{1}{2}}U^{*}(M_{e^{-2\pi i x}}(\mathbb{1}_{v}))=M_{e^{-2\pi i x}}(\mathbb{1}_{\mu})$$ from assuming $(4)$, we have
\begin{equation}\label{4to2}
(S^{\frac{1}{2}}U^{*}M_{e^{2\pi i x}}US^{-\frac{1}{2}})f=e^{2\pi ix}f\end{equation} for all $f\in L^{2}(\mu)$. Now by Theorem \ref{kaczmarzclass}, $S^{\frac{1}{2}}$ commutes with $U^{*}M_{e^{2\pi i x}}U$, meaning it also commutes with multiplication by $e^{2\pi i x}$ by equation (\ref{4to2}). Then $S$ is multiplication operator.

$(1),(2)\implies (4)$

We have as a consequence of statement $(1)$,
$$\langle S^{-\frac{1}{2}}(S(\mathbb{1})h_{n}),S^{-\frac{1}{2}}(S(\mathbb{1})h_{k})\rangle_{\mu}=\langle S(\mathbb{1})h_{n},h_{k}\rangle_{\mu}=\langle h_{n},h_{k}\rangle_{S(\mathbb{1})\mu}$$ for all $n$ and $k$.
Then it must be that from Corollary \ref{eqcor},
$$US^{-\frac{1}{2}}(S(h_{n}))=h_{n}\implies U^{*}h_{n}=S^{\frac{1}{2}}h_{n}$$ for all $n$.
Furthermore, $$\sum_{n=0}^{k}\langle e^{-2\pi i x},h_{n}\rangle_{S(\mathbb{1})\mu} h_{n}\to e^{-2\pi ix}$$ in $L^{2}(S(\mathbb{1})\mu)$ and $L^{2}(\mu)$ since $S(\mathbb{1})$ is bounded above and below as a consequence of statement $(1)$. Therefore, $(4)$ follows.
\end{proof}

\section{Necessity result for two-sided operator orbit Fourier expansions}\label{S4}

In \cite{Berner2024Frame}, we discussed measures $\mu$ on $[0,1)$ that possess sequences $\{g_{n}\}_{n\in \mathbb{Z}}$ that are Poisson dextrodual to $\{e^{2\pi i nx}\}_{n\in \mathbb{Z}}$. In this section, we are interested in the case when $\{g_{n}\}$ arises from an two-sided operator orbit. 
It is known that $\{e^{2\pi i nx}\}_{n\in \mathbb{Z}}$ forms a frame in $L^{2}(\mu)$ exactly when $\mu$ is absolutely continuous with Radon-Nikodym derivative bounded above and below on its support. This follows from work of He, Lai, and Lau in \cite{Lai2011Fourier}, \cite{He2013Exponential}. In this case, the frame operator associated with $\{e^{2\pi i nx}\}_{n\in \mathbb{Z}}$ is multiplication by the Radon-Nikodym derivative of $\mu$. Therefore, there is a frame of the form $\{T^{n}g_{0}\}_{n\in \mathbb{Z}}$ for a bounded invertible operator $T$ on $L^{2}(\mu)$ and $g_{0}\in L^{2}(\mu)$ that is Poisson dextrodual to $\{e^{2\pi i nx}\}_{n\in \mathbb{Z}}$. However in the following section, we show that $\{e^{2\pi i nx}\}_{n\in \mathbb{Z}}$ forming a frame in $L^{2}(\mu)$ is not necessary for $\mu$ possessing a sequence of the form $\{T^{n}g_{0}\}_{n\in \mathbb{Z}}$ that is Poisson dextrodual to $\{e^{2\pi i nx}\}_{n\in \mathbb{Z}}$.

Our first result is proving necessary conditions for measures that possess a sequence of the form $\{T^{n}g_{0}\}_{n\in \mathbb{Z}}$ that is Poisson dextrodual to $\{e^{2\pi inx}\}_{n\in \mathbb{Z}}$. The first condition we prove is a similar argument to the argument that Christensen and Hasannasab provided in Lemma 3.4 of \cite{Christensen2017Operator}, except our limit is taken over symmetric sums. Therefore, we added a reasonable coefficients decay condition to the following theorem. Furthermore, the fourth condition in the following theorem is inspired from the argument of Hunt, Muckenhoupt, and Wheeden in Theorem 8 of \cite{Hunt1973Weighted}.

\begin{thm}\label{Mthm}
Let $\mu$ be a finite Borel measure on $[0,1)$ such that there exists an invertible $T\in B(L^{2}(\mu))$ and $g_{0}\in L^{2}(\mu)$ where  $$\lim_{|M|\to \infty} \langle f,T^{M}g_{0}\rangle_{\mu} =0$$ for all $f\in L^{2}(\mu)$, and
$\{T^{n}g_{0}\}_{n\in \mathbb{Z}}$ is Poisson dextrodual to $\{e^{2\pi i nx}\}_{n\in \mathbb{Z}}$.
Then 
\begin{enumerate}
\item $T=M_{e^{2\pi ix}}$.
\item $\mu$ is absolutely continuous with Radon-Nikodym derivative $g$.
\item $g_{0}=\frac{1}{g}$ $\mu$ almost everywhere.
\item There exists $C>0$ such that for any interval $I\subseteq [0,1)$,
\begin{equation}\label{Meq}
\frac{1}{|I|}\int_{I}gdx \frac{1}{|I|}\int_{I}\frac{1}{g}\chi_{\{x: g(x)>0\}}dx \leq C.
\end{equation}
\end{enumerate}
\end{thm}
However, to prove this theorem, we need a lemma about the Dirichlet kernel where for each $M\in \mathbb{N}$, $$D_{M}(t)=\frac{sin(\pi(2M+1)t)}{sin(\pi t)}.$$ 

\begin{lem}\label{DirchletL}
For $N\geq 1$,
$D_{N}(t)\geq N$ whenever $|t|\leq \frac{1}{8N}$.
\end{lem}
This result can be proved using a calculus argument.
\begin{proof}[Proof of Theorem \ref{Mthm}]

$(1)$

Consider the following for any $f\in L^{2}(\mu)$,
\begin{equation}
\begin{split}
e^{2\pi i x}T^{*}f&=\lim\sum_{n=-M}^{M}\langle f, T^{n+1}g_{0}\rangle e^{2\pi i (n+1)x}\\
&=f-\lim_{M\to \infty} [\langle f,T^{(-M)}g_{0}\rangle e^{-2\pi i Mx}]+\lim_{M\to \infty}[\langle f,T^{(M+1)}g_{0}\rangle e^{2\pi i (M+1)x}]\\
&=f.
\end{split}
\end{equation}
Therefore, $T=M_{e^{2\pi ix}}$.

$(2)$

Let $f=\chi_{S}$ where $S$ is a Lebesgue measure zero set such that $\mu_{s}(S^{c})=0$ where $\mu_{s}$ is the singular part of $\mu$.
We have by assumption and from statement $(1)$ of Theorem \ref{Mthm},

\begin{equation}\label{abscp}
\chi_{S}=\lim\sum_{n=-M}^{M}\langle \chi_{S}, g_{0}e^{2\pi i nx}\rangle_{\mu}e^{2\pi i nx}=\lim\sum_{n=-M}^{M}\langle \overline{g_{0}}, e^{2\pi i nx}\rangle_{\mu_{s}}e^{2\pi i nx}\implies 
\end{equation}
$$
\langle \lim\sum_{n=-M}^{M}\langle \overline{g_{0}}, e^{2\pi i nx}\rangle_{\mu_{s}}e^{2\pi i nx}, \overline{g_{0}} \chi_{S}\rangle_{\mu}< \infty.
$$

That is,
$$\sum_{n=-\infty}^{\infty}|\langle \overline{g_{0}}, e^{2\pi inx}\rangle_{\mu_{s}}|^{2}<\infty.$$ Therefore, the complex measure $\overline{g_{0}}\mu_{s}$ has the same Fourier coefficients as an absolutely continuous complex measure with Radon-Nikodym derivative in $L^2([0,1)).$ However, since the complex measure  $\overline{g_{0}}\mu_{s}$ is also singular, it must be the case that 
$g_{0}$ is zero $\mu_{s}$ almost everywhere. Then by equation (\ref{abscp}), $\chi_{S}=0$ in $L^{2}(\mu)$, showing $\mu$ is absolutely continuous.

$(3)$

Let $g$ be the Radon-Nikodym derivative of $\mu$.
We have by the Poisson dextrodual property and statement $(1)$,
$$\langle 1, \overline{g_{0}}\rangle_{\mu} =\lim\sum_{n=-M}^{M}\langle 1, g_{0}e^{2\pi i nx}\rangle_{\mu} \langle g_{0}e^{2\pi i nx},1\rangle_{\mu} <\infty$$
$$\implies 
\sum_{n\in \mathbb{Z}}|\langle \overline{g_{0}}g,e^{2\pi inx}\rangle_{L^2([0,1))}|^{2}=\int_{0}^{1}|g_{0}|^{2}g^{2}dx<\infty.$$

Then we have $$\sum_{n=-M}^{M}\langle 1, g_{0}e^{2\pi inx}\rangle_{\mu} e^{2\pi inx}=\sum_{n=-M}^{M}\langle \overline{g_{0}}g, e^{2\pi inx}\rangle_{L^2([0,1))} e^{2\pi inx} \to \overline{g_{0}}g$$ point-wise almost everywhere by Carleson's theorem. However, there exists a sub-sequence of $\sum_{n=-M}^{M}\langle 1, g_{0}e^{2\pi inx}\rangle_{\mu} e^{2\pi inx}$ converging to $1$ point-wise $\mu$ almost everywhere.
Therefore, since $\mu$ is absolutely continuous, $$\overline{g_{0}}g=1$$ $\mu$ almost everywhere.

$(4)$

It suffices to prove the result for intervals of length at most $\frac{1}{8}$.
Now for each $M$, define the bounded operators $R_{M}: L^{2}(\mu)\to L^{2}(\mu)$ where
\begin{equation}\label{Fourierform}
R_{M}(f)=\sum_{n=-M}^{M}\langle f, T^{n}g_{0}\rangle_{\mu} e^{2\pi inx}=\sum_{n=-M}^{M}\langle f, \chi_{\{x: g(x)>0\}}e^{2\pi i nx}\rangle_{L^2([0,1))}e^{2\pi i nx}.
\end{equation} 

Note the equality in equation (\ref{Fourierform}) follows from statement (3). We also know by the Poisson dextrodual condition that for each $f\in L^{2}(\mu)$, $R_{M}(f)$ is a bounded sequence, making the $R_{M}$ uniformly bounded by the Principle of Uniform Boundedness with some bound $B$. 

Let $I$ be an interval in $[0,1)$ with $|I|\leq \frac{1}{8}$ and take $N=\lfloor \frac{1}{8|I|} \rfloor$ so that $|I|\leq \frac{1}{8N}$.
We have for $\theta \in I$ by Lemma \ref{DirchletL},
$$|R_{N}(\frac{1}{g}\chi_{\{x\in I: g(x)>0\}})(\theta)|=|\int_{I}\frac{1}{g(x)}\chi_{\{x: g(x)>0\}}D_{N}(\theta-x)dx|$$
$$\geq N \int_{I}\frac{1}{g(x)}\chi_{\{x: g(x)>0\}}dx.$$ Now using the fact that the $R_{M}$ operators are uniformly bounded by $B$ we get,
$$N^{2}\left [\int_{I}\frac{1}{g(x)}\chi_{\{x: g(x)>0\}}dx\right ]^{2}\int_{I}g(x)dx\leq B \int_{I}\frac{1}{g(x)}\chi_{\{x: g(x)>0\}}dx.$$
Finally, some simple algebra finishes the proof.

\end{proof}

\begin{conj}
If $\mu$ is a finite absolutely continuous Borel measure on $[0,1)$ whose Radon-Nikodym derivative $g$ satisfies equation $(\ref{Meq})$, then $\mu$ possesses a sequence of the form $\{T^{n}g_{0}\}$ that is Poisson dextrodual to $\{e^{2\pi i nx}\}_{n\in \mathbb{Z}}$.
\end{conj}

Since we were unable to prove this conjecture, the next section of this paper will be devoted to proving a partial converse.

\section{A sufficiency result}\label{S5}

For the converse result, we use Theorem \ref{Mthm}. Asking if an absolutely continuous measure $\mu$ with Radon-Nikodym derivative $g$ that satisfies equation $(\ref{Meq})$ possesses a sequence of the form $\{T^{n}g_{0}\}$ that is Poisson dextrodual to $\{e^{2\pi i nx}\}_{n\in \mathbb{Z}}$ is the same as asking if the sequence $$S_{n}(f\chi_{\{x: g(x)>0\}})\to f$$ in $L^{2}(\mu)$ for all $f\in L^{2}(\mu)$ where $S_{n}(f\chi_{\{x: g(x)>0\}})$ denotes the classical Fourier sequence of $f\chi_{\{x: g(x)>0\}}$. This problem is similar to the problem of weighted norm convergence of Fourier series that is discussed by Hunt, Muckenhoupt, and Wheeden \cite{Hunt1973Weighted}. They showed that weighted norm convergence of Fourier series happens only if the weight satisfies the $A_{2}$ condition, which is similar to satisfying equation (\ref{Meq}).
We state the $A_{2}$ condition here as well as the result in \cite{Hunt1973Weighted}.
\begin{defn}
A non-negative measurable function $w$ defined on $\mathbb{R}$ satisfies the $A_{2}$ condition if there is a $C>0$ such that
$$\left(\frac{1}{|I|}\int_{I}w(x)dx\right)\left(\frac{1}{|I|}\int_{I}\frac{1}{w(x)}dx\right)\leq C$$ for all intervals $I$.
\end{defn}

\begin{thm}[Theorem 8 \cite{Hunt1973Weighted}]\label{Hunthm}
Suppose that $w$ is a non-negative function with period $2\pi$. Then $w$ satisfies the $A_{2}$ condition if and only if for any $f$ of period $2\pi$ such that $\int_{-\pi}^{\pi}|f(x)|^{2}w(x)dx<\infty$, 
$$\int_{-\pi}^{\pi}|S_{n}(f)(x)-f(x)|^{2}w(x)dx\to 0$$ as $n\to \infty$ where $S_{n}(f)$ denotes the classical Fourier partial sums of $f$.
\end{thm}

It seems plausible that if the Radon-Nikodym derivative $g$ of $\mu$ satisfies equation $(\ref{Meq})$, this would imply $$S_{n}(f\chi_{\{x: g(x)>0\}})\to f$$ in $L^{2}(\mu)$ for all $f\in L^{2}(\mu)$. However, it is clear that satisfying equation $(\ref{Meq})$ is strictly weaker than satisfying the $A_{2}$ condition, and it is not clear that a weight satisfying $(\ref{Meq})$ can even be extended to a weight satisfying the $A_{2}$ condition. 
Therefore, to establish a sufficiency result, we draw on results of Kurki and Mudarra where in \cite{KurkiExtension2022} they establish results about when weights possess $A_{2}$ extensions.

\begin{thm}[Theorem 1.1 \cite{KurkiExtension2022}]\label{extthm}
Let $X$ be a closed interval, $E\subseteq X$ where $|E|>0$, and $w$ be a weight on $E$. The following are equivalent:
\begin{enumerate}
    \item There is a weight $W$ on $X$ satisfying the $A_{2}$ condition such that $W=w$ almost everywhere on $E$.
    \item There is an $\epsilon>0$ and $B>0$ such that for any interval $I\subseteq X$,
    $$\left(\frac{1}{|I|}\int_{I\cap E}w^{1+\epsilon}\right)\left(\frac{1}{|I|}\int_{I\cap E}\frac{1}{w^{1+\epsilon}}\right)\leq B.$$
\end{enumerate}
\end{thm}

This gives us a sufficiency result for this section.
\begin{cor}\label{suffcase}
Let $\mu$ be an absolutely continuous finite Borel measure on $[0, 1)$ with Radon-Nikodym derivative $g$ such that there exists an $\epsilon>0$ where $g^{1+\epsilon}$ satisfies equation (\ref{Meq}).
Then, there exists an invertible $T\in
B(L^{2}(\mu))$ and $g_{0} \in L^{2}(\mu)$ such that $\{T^{n}g_{0}\}_{n\in \mathbb{Z}}$ is Poisson dextrodual to $\{e^{2\pi i nx}\}$.
\end{cor}
\begin{proof}
By Theorem \ref{extthm}, there is a $G$ on $[0,1)$ satisfying the $A_{2}$ condition where $G=g$ almost everywhere on $\{x: g(x)>0\}$.
The result follows from Theorem \ref{Hunthm} and Theorem \ref{Mthm} by taking functions $f$ on $[0,1)$ that are $\mu$ square integrable and vanish outside of $\{x: g(x)>0\}$.
\end{proof}

\begin{rem}
Lebesgue measure with unbounded weight $\frac{1}{\sqrt{x}}$ is a simple example of a measure that
possesses a sequence of the form $\{T^{n}g_{0}\}$ that is Poisson dextrodual to $\{e^{2\pi i nx}\}_{n\in \mathbb{Z}}$ since weight $\frac{1}{\sqrt{x}}$ satisfies the $A_{2}$ condition.
\end{rem}

\section{classifying the frame case.}\label{S6}

In this section, we classify measures that possess frames of the form $\{T^{n}g_{0}\}_{n\in \mathbb{Z}}$ that are Poisson dextrodual to $\{e^{2\pi inx}\}_{n\in \mathbb{Z}}$. The first theorem is one direction of the classification.

\begin{thm}\label{Lfthm}
Under the hypothesis of Theorem \ref{Mthm}, the following are equivalent:
\begin{enumerate}
    \item $\frac{1}{g_{0}}\in L^{\infty}(\mu)$.
    \item $\{e^{2\pi i nx}\}_{n\in \mathbb{Z}}$ is a Bessel sequence in $L^{2}(\mu)$.
    \item $\{T^{n}g_{0}\}$ is a lower semi-frame in $L^{2}(\mu)$.
\end{enumerate}
\end{thm}
\begin{proof}
$(1)\implies (2)$

For all $f\in L^{2}(\mu)$, $\frac{f}{g_{0}}\in L^{2}(\mu)$, giving us that
$\lim \sum_{n=-M}^{M}\langle f, e^{2\pi i nx}\rangle_{\mu} e^{2\pi i nx}$ exists for all $f\in L^{2}(\mu)$ by Theorem \ref{Mthm}. This implies
$$\langle \lim \sum_{n=-M}^{M}\langle f, e^{2\pi i nx}\rangle_{\mu} e^{2\pi i nx},f\rangle_{\mu} =\sum_{n\in \mathbb{Z}}|\langle f, e^{2\pi inx}\rangle_{\mu} |^{2}<\infty$$ for all $f\in L^{2}(\mu)$. Therefore, the operator $f\to \{\langle f, e^{2\pi inx}\rangle_{\mu}\}_{n\in \mathbb{Z}} $ is a point-wise limit of bounded operators and is therefore bounded.

$(2)\implies (3)$

For any $f\in L^{2}(\mu)$ we have the following by Cauchy-Schwarz,
$$\|f\|^{2}=\lim\sum_{n=-M}^{M}\langle f, T^{n}g_{0}\rangle \langle e^{2\pi i nx}, f\rangle \leq (\|\{\langle f, T^{n}g_{0}\rangle\}\|_{\ell^{2}})\sqrt{B}\|f\|$$
for some $B$ by the Bessel condition. Dividing the last and first term by $\|f\|$ gets the result.

$(3) \implies (1)$

Let $f\in L^{\infty}(\mu)$, we have by the lower semi-frame condition and Cauchy-Schwarz,
$$\langle \lim\sum_{n=-M}^{M}\langle f, g_{0}e^{2\pi i nx}\rangle e^{2\pi i nx},\overline{g_{0}} f\rangle=\sum_{n}|\langle f,g_{0}e^{2\pi i nx}\rangle|^{2}\leq $$
$$\sqrt{B}\|\{\langle f, g_{0}e^{2\pi i nx}\rangle \}\|_{\ell^{2}} \|g_{0}f\|$$ for some $B$.
Then we have
\begin{equation}\label{bbi}\frac{1}{B}\|f\|^{2}\leq \sum_{n}|\langle f,g_{0}e^{2\pi i nx}\rangle|^{2}\leq B \|g_{0}f\|^{2}.\end{equation}

Suppose that $g_{0}<\frac{1}{2B}$ on a Borel set $E\subseteq [0,1)$. By letting $f=\chi_{E}$ in equation (\ref{bbi}),
$$\mu(E)\leq B^{2}\int_{E}|g_{0}|^{2}d\mu\leq \frac{1}{4}\mu(E)\implies \mu(E)=0.$$
Therefore, $$g_{0}\geq \frac{1}{2B}$$ $\mu$ almost everywhere.
\end{proof}

The following theorem provides the other direction for classification.

\begin{thm}\label{Besselthm}
Under the hypothesis of Theorem \ref{Mthm}, the following are equivalent:
\begin{enumerate}
    \item $g_{0}\in L^{\infty}(\mu)$.
    \item $\{T^{n}g_{0}\}$ is a Bessel sequence in $L^{2}(\mu)$.
    \item $\{e^{2\pi i nx}\}_{n\in \mathbb{Z}}$ is a lower semi-frame in $L^{2}(\mu)$.
\end{enumerate}
\end{thm}

\begin{proof}
$(1)\implies (2)$

For any $f\in L^{2}(\mu)$, by Cauchy-Schwarz we have
$$\langle \lim \sum_{n=-M}^{M}\langle f, g_{0}e^{2\pi i nx}\rangle e^{2\pi i nx}, fg_{0}\rangle =\sum_{n}|\langle f, g_{0}e^{2\pi i nx}\rangle |^{2}\leq \|f\| \ \|fg_{0}\|\leq \|g_{0}\|_{\infty} \ \|f\|^{2}.$$

$(2)\implies (3)$

This argument is the same as $(2)\implies (3)$ in Theorem \ref{Lfthm}.

$(3)\implies (1)$

Suppose for the sake of contradiction that $g(x)\leq \frac{B}{3}$ on a set of positive Lebesgue measure $S\subseteq \{x: g(x)>0\}$ where $B$ is the lower semi-frame bound from $(3)$. We have
$$S=\bigcup_{n=1}^{\infty}E_{n}$$ where $E_{n}=\{x\in S : \frac{\frac{B}{3}}{n+1}<g(x)\leq \frac{\frac{B}{3}}{n}\}$. There exists a $k\geq 1$ such that $E_{k}$ has positive Lebesgue measure. We have
$$\frac{(\frac{B}{3})B}{(k+1)}|E_{k}|\leq B\int_{E_{k}}gdx\leq \sum_{n\in \mathbb{Z}}|\langle \chi_{E_{k}},e^{2\pi i nx}\rangle_{\mu}|^{2}$$
$$=\int_{E_{k}}g^{2}dx\leq \frac{(\frac{B}{3})^{2}}{k^{2}}|E_{k}|\implies
3\leq \frac{k+1}{k^{2}},$$ which is a contradiction. Then $g(x)$ is bounded below on its support, making $g_{0}$ bounded $\mu$ almost everywhere by Theorem \ref{Mthm}.
\end{proof}
To establish our full classification result, we use the result of Lai:

\begin{thm}[Theorem 1.1 \cite{Lai2011Fourier}]\label{absc}
    If $\mu$ is an absolutely continuous Borel probability measure on $[0,1)$, then $\mu$ possesses a frame of exponentials for $L^{2}(\mu)$ if and only if the Radon-Nykodym derivative of $\mu$ is bounded above and below on its support.
\end{thm}

Furthermore, we show unconditional convergence of Fourier series of this type can only be achieved with true frame series. This follows from results of Heil and Yu in \cite{Heil2023Convergence} and Gohberg in \cite{Gohberg1969Theory}:

\begin{lem}[Lemma 2.7 \cite{Heil2023Convergence}]\label{Heil}
Let $\{x_{n}\}$ and $\{y_{n}\}$ be sequences in Hilbert space $H$. For any $f\in H$, the series $\sum_{n}\langle f, x_{n}\rangle y_{n}$ converges unconditionally if and only if for any $f\in H$, the series $\sum_{n}\langle f, y_{n}\rangle x_{n}$ converges unconditionally.
\end{lem}

\begin{thm}[Lemma 2.2 \cite{Gohberg1969Theory}]\label{Gohberg}
Suppose sequence $\{x_{n}\}$ in a Hilbert space $H$ has the series $\sum_{n}x_{n}$ converge unconditionally, then $$\sum_{n}\|x_{n}\|^{2}<\infty.$$
\end{thm}

\begin{cor}\label{classcor}
Under the hypothesis of Theorem \ref{Mthm}, the following are equivalent:
\begin{enumerate}
    \item $\{T^{n}g_{0}\}$ is a frame in $L^{2}(\mu)$
    \item $g_{0},\frac{1}{g_{0}}\in L^{\infty}(\mu).$
    \item $\{e^{2\pi i nx}\}_{n\in \mathbb{Z}}$ is a frame in $L^{2}(\mu)$.
    \item $g,\chi{\{x:g(x)>0\}}\frac{1}{g}\in L^{\infty}([0,1))$.
    \item For any bijection $\sigma: \mathbb{N}\to \mathbb{Z}$, $\sum_{n=0}^{\infty}\langle f, T^{\sigma(n)}g_{0}\rangle e^{2\pi i \sigma(n) x}$ exists in $L^{2}(\mu)$ for all $f\in L^{2}(\mu)$.
\end{enumerate}
\end{cor}
\begin{proof}
$(1)\iff (2) \iff (3)$ follows from Theorems \ref{Lfthm} and \ref{Besselthm}.

$(3)\iff (4)$ follows from Theorem \ref{absc}.

Assuming $(5)$, we get by Lemma \ref{Heil} and Theorem \ref{Gohberg} that for all $f\in L^{2}(\mu)$,
$$\sum_{n\in \mathbb{Z}}|\langle f, e^{2\pi i nx}\rangle|^{2}, \sum_{n\in \mathbb{Z}}|\langle f, g_{0}e^{2\pi i nx}\rangle|^{2}<\infty.$$
It follows that $\{T^{n}f_{0}\}$ and $\{e^{2\pi nx}\}$ are Bessel sequences in $L^{2}(\mu)$, so by Theorem \ref{Besselthm}, we have $(3)$.

Assume $(3)$. It is known that if $\{c_{n}\}\in \ell^{2}(\mathbb{Z})$ and $\{g_{n}\}_{n\in \mathbb{Z}}$ is a Bessel sequence in a Hilbert space $H$, then 
$\sum_{n}c_{n}g_{n}$ converges unconditionally, giving us $(3)\implies (5)$ by $(3)\iff (1)$.
\end{proof}

\section{Other classifying results}\label{S7}
In this section, we use previous results to establish more classification results for measures that possess Bessel sequences $\{g_{n}\}$ that are Poisson dextrodual to $\{e^{2\pi i nx}\}_{n\in \mathbb{Z}}$.
Our first result is that there are even absolutely continuous measures that don't possess these types of expansions:

\begin{prop}\label{abscontfail}
There exists an absolutely continuous measure $\mu$ on $[0,1)$ such that there is an $f\in L^{2}(\mu)$ where no $\{c_{n}\}\in \ell^{2}(\mathbb{Z})$ has 
$$f=\lim \sum_{n=-M}^{M}c_{n}e^{2\pi inx}$$ with convergence in $L^{2}(\mu)$.
\end{prop}
\begin{proof}
Let $f$ be a bounded Borel function on $[0,1)$ such that
$$\limsup_{M\to \infty}|S_{M}(f)(0)|=\infty$$ where $S_{M}(f)$ denotes the classic Fourier sequence of $f$. There is a $g\in L^{1}([0,1))$ where the sequence
$\int_{0}^{1}|S_{M}(f)|^{2}gdx$ is unbounded. Otherwise by Principle of uniform boundedness, we would have
$$sup_{n}\|S_{M}(f)\|_{\infty}<\infty.$$

Let $\mu$ be absolutely continuous with Radon-Nikodym derivative $|g|+1$. We claim $f$ can't be expressed as a Fourier series in $L^{2}(\mu)$. Note that $\{\frac{e^{2\pi i nx}}{\sqrt{1+|g|}}\}_{n\in \mathbb{Z}}$ is an orthonormal basis of $L^{2}(\mu)$.
Then if $$f=\lim \sum_{n=-M}^{M}c_{n}e^{2\pi i nx}\implies \frac{f}{\sqrt{1+|g|}}=\lim \sum_{n=-M}^{M}c_{n}\frac{e^{2\pi i nx}}{\sqrt{1+|g|}}$$
$$\implies c_{n}=\langle \frac{f}{\sqrt{1+|g|}}, \frac{e^{2\pi inx}}{\sqrt{1+|g|}}\rangle_{\mu}=\langle f, e^{2\pi inx}\rangle_{L^{2}([0,1))}$$ for all $n$ by uniqueness of coefficients from an orthonormal basis.
Then it must be that $\int_{0}^{1}|S_{M}(f)|^{2}(1+|g|)dx$ is a bounded sequence, which is a contradiction.
\end{proof}
Our next classification result follows from a couple of results from \cite{Berner2024Frame} and Theorem \ref{Hunthm}. We state the results from \cite{Berner2024Frame} here:

\begin{prop}[Proposition 3.2 \cite{Berner2024Frame}]\label{Lebesguepart}
    Let $\mu=\mu_{a}+\mu_{s}$ where $\mu_{s}$ is a non-zero finite singular Borel measure on $[0,1)$ and $\mu_{a}$ is an absolutely continuous Borel measure on $[0,1)$ with Radon-Nikodym derivative that is bounded below. For $S$ that is a Borel set of Lebesgue measure zero where $\mu_{s}(S^{c})=0$, there is no $\{c_{n}\}_{n=-\infty}^{\infty}$ such that
    $$\sum_{n}c_{n}e^{2\pi i nx}=\chi_{S}$$ with convergence in $L^{2}(\mu)$.
\end{prop}

\begin{thm}[Theorem 3.5 \cite{Berner2024Frame}]\label{oldthm}
Let $\mu$ be a finite Borel measure on $[0,1)$ with Radon-Nikodym  derivative $g$. Suppose that there is a Bessel sequence $\{g_{n}\}$ with bound $B$ that is Poisson dextrodual to $\{e^{2\pi i nx}\}_{n=-\infty}^{\infty}$ in $L^{2}(\mu)$, then $$g(x)>\frac{1}{3B}$$ almost everywhere on its support.
\end{thm}

\begin{thm}
Let $\mu$ be a finite Borel measure on $[0,1)$ with Radon-Nikodym derivative $g$.
Suppose there is a Bessel sequence $\{g_{n}\}$ that is Poisson dextrodual to $\{e^{2\pi i nx}\}_{n\in \mathbb{Z}}$ in $L^{2}(\mu)$. Then the following are equivalent:
\begin{enumerate}
    \item $\mu$ is absolutely continuous, $\frac{1}{g}\in L^{\infty}([0,1))$, and $g$ satisfies the $A_{2}$ condition.
    \item $g>0$ almost everywhere.
    \item $\{g_{n}\}$ has a biorthogonal sequence.
    \item $\{e^{2\pi i nx}\}_{n\in \mathbb{Z}}$ has a biorthogonal sequence.
\end{enumerate}
\end{thm}
\begin{proof}
$(1)\implies (2)$

Since $g$ satisfies the $A_{2}$ condition, $\frac{1}{g}<\infty$ almost everywhere.

$(2)\implies (1)$

By Theorem \ref{oldthm} and Proposition \ref{Lebesguepart}, $\mu$ is absolutely continuous and $\frac{1}{g}\in L^{\infty}([0,1))$. Furthermore since $\{\frac{e^{2\pi i nx}}{\sqrt{g}}\}_{n\in \mathbb{Z}}$ is an orthonormal basis in $L^{2}(\mu)$, by the proof of Proposition \ref{abscontfail}, it must be that
$$\langle f, g_{n}\rangle_{\mu}=\langle f, e^{2\pi i nx}\rangle_{L^{2}([0,1))}$$ for all $n\in \mathbb{Z}$. It follows that $g$ satisfies the $A_{2}$ condition by Theorem \ref{Hunthm}.

$(2)\implies (3)$

Continuing the argument $(2)\implies (1)$, we see that $\{e^{2\pi i nx}\}_{n\in \mathbb{Z}}$ is biorthogonal to $\{g_{n}\}$ in $L^{2}(\mu)$.

$(3)\implies (4)$

Let $\{g_{n}'\}$ be the biorthogonal sequence of $\{g_{n}\}$. By the Poisson dextrodual property, for each $k$,
$$g_{k}'=e^{2\pi i kx}.$$

$(4)\implies (2)$

Suppose for sake of contradiction that $g(x)=0$ for $x\in S$ where $|S|>0$.
Let $\{e_{n}'\}$ be the biorthogonal sequence of $\{e^{2\pi i nx}\}_{n\in \mathbb{Z}}$. Then for each $k$, the complex measure $e_{k}'\mu$ has the same Fourier coefficients as the complex measure $e^{2\pi i kx}dx$. Then it must be that  $e_{k}'\mu$ is absolutely continuous and $$e_{k}'(x)g(x)=e^{2\pi i kx}$$ for almost every $x$.
Therefore, for almost every $x\in S$, $$e^{2\pi i kx}=0,$$ which is a contradiction. 

\end{proof}

Finally, we show that if a measure is not equivalent to Lebesgue measure, and it possesses a frame that is Poisson dextrodual to the exponentials, then that frame has infinite excess. 
\begin{defn}
Recall a frame $\{g_{n}\}_{n}$ in Hilbert space $H$ has \textbf{finite excess} if there is a finite set $F$ such that $\{g_{n}\}_{n}-\{g_{n}\}_{n\in F}$ is not a frame.
\end{defn}
To prove this result, we use a result from Bakic and Beric about frame excess \cite{Bakic2015Excesses}. For more about frame excess and dual frames, see \cite{Christensen2016Frames}.

\begin{thm}[Theorem 2.2 \cite{Bakic2015Excesses}]\label{Bakicthm}
Let $\{x_{n}\}$ and $\{y_{n}\}$ be frames in a Hilbert space that are dual to each other. Then $\{x_{n}\}$ and $\{y_{n}\}$ have the same amount of excess.
\end{thm}

\begin{thm}\label{excess}
Let $\mu$ be a finite Borel measure on $[0,1)$ such that there exists a frame with finite excess $\{g_{n}\}$ that is Poisson dextrodual to $\{e^{2\pi i nx}\}_{n\in \mathbb{Z}}$ in $L^{2}(\mu)$.
Then $\mu$ is absolutely continuous with Radon-Nikodym derivative bounded above and below.
\end{thm}
\begin{proof}
It is known that $$\{g_{n}\}_{n}=\{g_{n}\}_{n\in F}\cup \{g_{n}\}_{n\in F^{c}}$$ where $F$ is finite and $\{g_{n}\}_{n\in F^{c}}$ is a Riesz basis. For details see Theorem 5.4.7 of \cite{Christensen2016Frames}. Let $S$ be the frame operator of $\{g_{n}\}_{n\in F^{c}}$, and choose $k$ so that $$\{g_{n}\}_{|n|\geq k}\subseteq \{g_{n}\}_{n\in F^{c}}.$$
See that for any $\{c_{n}\}\in \ell^{2}(\mathbb{Z})$ where $c_{n}=0$ for $|n|<k$, we have
$$\lim \sum_{n=-M}^{M}\langle \sum_{|m|\geq k}c_{m}S^{-1}g_{m}, g_{n}\rangle e^{2\pi i nx}=$$
$$\lim \sum_{n=-(M+k)}^{M+k}c_{n}e^{2\pi i nx}+\sum_{|n|<k}\langle \sum_{|m|\geq k}c_{m}S^{-1}g_{m}, g_{n}\rangle e^{2\pi i nx}$$
since $\{S^{-1}g_{n}\}_{n\in F^{c}}$ is biorthogonal to $\{g_{n}\}_{n\in F^{c}}$.
It follows by the Poisson dextrodual property that for any $\{c_{n}\}\in \ell^{2}(\mathbb{Z})$, $\sum_{n=-M}^{M}c_{n}e^{2\pi i nx}$ converges in $L^{2}(\mu)$. Therefore, the synthesis operator associated with $\{e^{2\pi i nx}\}_{n\in \mathbb{Z}}$ is bounded from being a point-wise limit of bounded operators. It follows from the same argument in Theorem \ref{Besselthm} that  $\{e^{2\pi i nx}\}_{n\in \mathbb{Z}}$ is a frame. Furthermore, $\{e^{2\pi i nx}\}_{n\in \mathbb{Z}}$ has finite excess by Theorem \ref{Bakicthm}. We also have by Theorem \ref{absc} that $\mu$ is absolutely continuous with Radon-Nikodym derivative bounded above and below on its support.

Let $g$ be the Radon-Nikodym derivative of $\mu$. Suppose for sake of contradiction that $$g(x)=0$$ for $x\in S$ where $|S|>0$. Let $f\in L^{2}([0,1))$ and vanish outside of $S$.
We have $$\sum_{n=-M}^{M}\langle f, e^{2\pi i nx}\rangle_{L^{2}([0,1))}e^{2\pi inx}\to 0$$ in $L^{2}(\mu)$ since $g$ is bounded. Therefore, the kernel of the synthesis operator associated with $\{e^{2\pi i nx}\}_{n\in \mathbb{Z}}$ is infinite dimensional, which contradicts $\{e^{2\pi i nx}\}_{n\in \mathbb{Z}}$ having finite excess, see \cite{Bakic2015Excesses}.
\end{proof}

\section*{Acknowledgment} This research was supported in part  by the National Science Foundation and the National Geospatial Intelligence Agency under awards \#1830254 and \#2219959.

\printbibliography
\end{document}